\documentclass[11pt]{amsart}

\usepackage{amscd}
\usepackage{amsmath, amssymb}
\usepackage{amsfonts}
\newcommand{\de}{\partial}

\newcommand{\ov}[1]{\overline{#1}}

\newcommand{\ti}[1]{\tilde{#1}}
\newcommand{\vp}{\varphi}

\newcommand{\ve}{\varepsilon}

\renewcommand{\leq}{\leqslant}
\renewcommand{\geq}{\geqslant}

\newcommand{\be}{\begin{equation}}
\newcommand{\ee}{\end{equation}}

\begin{document}
\newcounter{remark}
\newcounter{theor}
\setcounter{remark}{0}
\setcounter{theor}{1}
\newtheorem{claim}{Claim}
\newtheorem{theorem}{Theorem}[section]
\newtheorem{lemma}[theorem]{Lemma}
\newtheorem{corollary}[theorem]{Corollary}
\newtheorem{proposition}[theorem]{Proposition}
\newtheorem{question}{question}[section]
\newtheorem{defn}{Definition}[theor]
\numberwithin{equation}{section}

\title{$C^{1,1}$ regularity of geodesics in the space of volume forms}

\author{Jianchun Chu}
\address{Institute of Mathematics, Academy of Mathematics and Systems Science, Chinese Academy of Sciences, Beijing 100190, P. R. China}
\email{chujianchun@gmail.com}

\subjclass[2010]{Primary: 58E10; Secondary: 58D17, 35J60, 35J70}

\begin{abstract}
We prove a $C^{1,1}$ estimate for solutions of a class of fully nonlinear equations introduced by Chen-He. As an application, we prove the $C^{1,1}$ regularity of geodesics in the space of volume forms.
\end{abstract}

\maketitle

\section{Introduction}
Let $(M,g)$ be a Riemannian manifold of real dimension $n$. We use $\nabla$ to denote the Levi-Civita connection. Recently, Chen-He \cite{CH18} introduced the following function space
\begin{equation*}
\ti{\mathcal{H}} = \{\vp\in C^{\infty}(M)~|~\Delta\vp-b|\nabla\vp|^{2}+a(x)>0\},
\end{equation*}
where $b$ is a nonnegative constant and $a(x)$ is a positive smooth function on $M$. For any $u_{0},u_{1}\in\ti{\mathcal{H}}$, they also introduced the fully nonlinear equation
\begin{equation}\label{Chen-He equation}
u_{tt}(\Delta u-b|\nabla u|^{2}+a(x))-|\nabla u_{t}|^{2} = f,
\end{equation}
with boundary condition
\begin{equation}\label{Boundary condition 1}
u(\cdot,0) = u_{0}, \,\,\, u(\cdot,1) = u_{1},
\end{equation}
where $f$ is a nonnegative function on $M\times[0,1]$. In \cite{CH18}, Chen-He solved the equation (\ref{Chen-He equation}) with uniform weak $C^{2}$ estimates, which also hold for the degenerate case (see also \cite{He08}).

When $b=0$, $a=1$ and $f=0$, (\ref{Chen-He equation}) becomes the geodesic equation in the space of volume forms on $(M,g)$. More specifically, in \cite{Donaldson10}, Donaldson introduced a Weil-Peterson type metric on the space of volume forms (normalized) on any Riemannian manifold with fixed total volume. We write $\mathcal{H}$ for this infinite dimensional space, which can be parameterized by the space of smooth functions
\begin{equation*}
\{\vp\in C^{\infty}(M)~|~1+\Delta\vp>0\}.
\end{equation*}
For any $\vp\in\mathcal{H}$, the tangent space $T_{\vp}\mathcal{H}$ is $C^{\infty}(M)$. And the metric is defined by
\begin{equation*}
\|\delta\vp\|_{\vp}^{2} = \int_{M}|\delta\vp|^{2}(1+\Delta\vp)dV_{g} \,\,\,\, \text{for $\delta\vp\in T_{\vp}\mathcal{H}$.}
\end{equation*}
For a path $\Phi:[0,1]\rightarrow\mathcal{H}$, the energy function is given by
\begin{equation*}
E(\Phi) = \int_{0}^{1}\int_{M}|\dot{\Phi}|^{2}(1+\Delta\Phi)dV_{g}
\end{equation*}
and the geodesic equation is
\begin{equation}\label{Geodesic equation}
\Phi_{tt}(1+\Delta\Phi)-|\nabla\Phi_{t}|^{2} = 0,
\end{equation}
with boundary condition
\begin{equation*}
\Phi(\cdot,0) = \vp_{0}, \,\,\, \Phi(\cdot,1) = \vp_{1},
\end{equation*}
where $\vp_{0},\vp_{1}\in\mathcal{H}$.

To solve this equation, for any $\ve>0$, Donaldson \cite{Donaldson10} introduced the following perturbed geodesic equation
\begin{equation}\label{Perturbed geodesic equation}
(\Phi_{\ve})_{tt}(1+\Delta\Phi_{\ve})-|\nabla(\Phi_{\ve})_{t}|^{2} = \ve,
\end{equation}
with boundary condition
\begin{equation}\label{Boundary condition 2}
\Phi_{\ve}(\cdot,0) = \vp_{0}, \,\,\, \Phi_{\ve}(\cdot,1) = \vp_{1}.
\end{equation}
In \cite{CH11}, Chen-He solved this perturbed geodesic equation and proved weak $C^{2}$ estimate which is independent of $\ve$. Let $\ve\rightarrow0$. Chen-He proved that there is a unique weak geodesic $\Phi$ connecting $\vp_{0}$ and $\vp_{1}$, and that the quantities $\sup_{M\times[0,1]}|\Phi|$, $\sup_{M\times[0,1]}|\Phi_{t}|$, $\sup_{M\times[0,1]}|\nabla\Phi|$, $\sup_{M\times[0,1]}|\Phi_{tt}|$, $\sup_{M\times[0,1]}|\nabla\Phi_{t}|$, $\sup_{M\times[0,1]}|\Delta\Phi|$ are all bounded (see \cite[Theorem 1.2, Corollary 5.3]{CH11}). By the boundary condition (\ref{Boundary condition 2}), the quantity $\sup_{\de(M\times[0,1])}|\nabla^{2}\Phi|$ is also bounded. Hence, $\Phi$ is $C^{1,\alpha}$ for any $\alpha\in(0,1)$.

In general, it is well known that the weak geodesic $\Phi$ is not $C^{2}$. Actually, in complex dimension $1$, (\ref{Geodesic equation}) becomes the geodesic equation in the space of K\"{a}hler metrics. And there are many examples which show that in general the weak geodesic in the space of K\"{a}hler metrics is not $C^{2}$ (see \cite{LV13,DL12,Darvas14}). Recently, Chu-Tosatti-Weinkove \cite{CTW17} proved the $C^{1,1}$ regularity of geodesics in the space of K\"{a}hler metrics.

Hence, for (\ref{Perturbed geodesic equation}), it was expected that $\sup_{M\times[0,1]}|\nabla^{2}\Phi_{\ve}|\leq C$, where $C$ is independent of $\ve$. This implies that the weak geodesic $\Phi$ is $C^{1,1}$. In this paper, we prove the $C^{1,1}$ regularity of geodesics in the space of volume forms.

\begin{theorem}\label{C11 regularity of geodesics}
Let $(M,g)$ be a compact $n$-dimensional Riemannian manifold. For any two points $\vp_{0},\vp_{1}\in\mathcal{H}$, the weak geodesic $\Phi$ connecting them is $C^{1,1}$.
\end{theorem}

As alluded to above, Theorem \ref{C11 regularity of geodesics} is a consequence of \cite[Theorem 1.2]{CH11} and the $C^{1,1}$ estimate for (\ref{Perturbed geodesic equation}). More generally, for (\ref{Chen-He equation}), Chen-He expected that $\sup_{M\times[0,1]}|\nabla^{2}u|$ is bounded (see \cite[Remark 2.15]{CH18}). We prove the following $C^{1,1}$ estimate, which confirms what Chen-He suggested.

\begin{theorem}\label{Main estimate}
Let $(M,g)$ be a compact $n$-dimensional Riemannian manifold. Suppose that $f$ is a positive smooth function on $M\times[0,1]$. For any smooth solution $u$ of (\ref{Chen-He equation}) satisfying
\begin{equation*}
u(\cdot,t) \in \ti{\mathcal{H}} ~\text{~for $t\in [0,1]$},
\end{equation*}
there exists a constant $C$ depending only on $\sup_{M\times[0,1]}|\nabla u|$, $\sup_{M\times[0,1]}|u_{tt}|$,  $\sup_{M\times[0,1]}|\Delta u|$, $\sup_{M\times[0,1]}f$, $\sup_{M\times[0,1]}|\nabla(f^{\frac{1}{2}})|$, $\sup_{M\times[0,1]}|\nabla^{2}(f^{\frac{1}{2}})|$, $u_{0}$, $u_{1}$, $a$, $b$ and $(M,g)$, such that
\begin{equation}\label{C11 estimate}
\sup_{M\times[0,1]}|\nabla^{2}u| \leq C.
\end{equation}
\end{theorem}

Combining this $C^{1,1}$ estimate, \cite[Theorem 1.1]{CH18} and the approximation argument, we obtain the following corollary.

\begin{corollary}\label{Corollary}
Let $(M,g)$ be a compact $n$-dimensional Riemannian manifold. Suppose that $f$ is a nonnegative function on $M$ such that
\begin{equation*}
\sup_{M\times[0,1]}\left(f+|(f^{\frac{1}{2}})_{t}|+|\nabla(f^{\frac{1}{2}})|+|f_{tt}|+|\nabla^{2}(f^{\frac{1}{2}})|\right) \leq C
\end{equation*}
for a constant $C$. Then the Dirichlet problem (\ref{Chen-He equation}) has a $C^{1,1}$ solution.
\end{corollary}

We note that (\ref{Chen-He equation}) also covers the Gursky-Streets equation when $k=1$ (see \cite{GS16}). Thus, Corollary \ref{Corollary} shows the existence of $C^{1,1}$ solutions to the Gursky-Streets equation ($k=1$).

\section{Proof of Theorem \ref{Main estimate}}
We use the same notations as in \cite{CH18}. For $r=(r_{0},r_{1},\cdots,r_{n+1})$, we write
\begin{equation*}
Q(r) = r_{0}r_{1}-\sum_{i=2}^{n+1}r_{i}^{2} \text{~and~} G(r) = \log Q(r).
\end{equation*}
We denote the first and second derivatives of $Q$ and $G$ by
\begin{equation*}
Q^{i} = \frac{\de Q}{\de r_{i}},
Q^{i,j} = \frac{\de^{2}Q}{\de r_{i}\de r_{j}},
G^{i} = \frac{\de G}{\de r_{i}},
G^{i,j} = \frac{\de^{2}G}{\de r_{i}\de r_{j}}.
\end{equation*}

For any point $x_{0}\in M$. Let $\{e_{i}\}_{i=1}^{n}$ be a local orthonormal frame in a neighborhood of $x_{0}$. In this paper, the subscripts of a function always denote the covariant derivatives. If we write $r=(u_{tt},B_{u},u_{ti})$ and $B_{u}=\Delta u-b|\nabla u|^{2}+a(x)$, then (\ref{Chen-He equation}) can be written as
\begin{equation}\label{Chen-He equation 1}
Q(r) = Q(u_{tt}, B_{u},u_{ti}) = u_{tt}B_{u}-|\nabla u_{t}|^{2} = f.
\end{equation}
Since $f>0$ and $u(\cdot,t)\in\ti{\mathcal{H}}$ for $t\in [0,1]$, we have $u_{tt}>0$ and $B_{u}>0$. By \cite[(2.8)]{CH18}, the linearized operator of $Q$ is given by
\begin{equation}\label{Definition of dQ}
dQ(\psi) = u_{tt}\left(\Delta\psi-2b(\nabla u,\nabla\psi)\right)+B_{u}\psi_{tt}-2(\nabla u_{t},\nabla\psi_{t}),
\end{equation}
where $(\cdot,\cdot)$ denotes the inner product. Clearly, the equation (\ref{Chen-He equation 1}) is elliptic.

Now we are in a position to prove Theorem \ref{Main estimate}.
\begin{proof}[Proof of Theorem \ref{Main estimate}]
Let $\lambda_{1}(\nabla^{2}u)$ be the largest eigenvalue of $\nabla^{2}u$. It is clear that
\begin{equation}\label{Main estimate equation 1}
|\nabla^{2}u| \leq C|\Delta u|+C\max\left(\lambda_{1}(\nabla^{2}u),0\right).
\end{equation}
To prove Theorem \ref{Main estimate}, it suffices to prove $\sup_{M\times[0,1]}\lambda_{1}(\nabla^{2}u)\leq C$. Hence, we consider the following quantity
\begin{equation*}
H(x,t,\xi) = u_{\xi\xi}+|\nabla u|^{2}+At^{2},
\end{equation*}
for $(x,t)\in M\times[0,1]$, $\xi\in T_{x}M$ a unit vector and $A$ a constant to be determined later. Let $(x_{0},t_{0},\xi_{0})$ be the maximum point of $H$. Without loss of generality, we assume that $(x_{0},t_{0})\notin\de(M\times[0,1])$. Otherwise, by the boundary condition (\ref{Boundary condition 1}), we obtain (\ref{C11 estimate}) directly. We choose a local orthonormal frame $\{e_{i}\}_{i=1}^{n}$ near $x_{0}$ such that
\begin{equation*}
e_{1}(x_{0}) = \xi_{0}.
\end{equation*}
In a neighborhood of $(x_{0},t_{0})$, we define a new quantity by
\begin{equation*}
\ti{H}(x,t) = H(x,t,e_{1}) = u_{11}+|\nabla u|^{2}+At^{2}.
\end{equation*}
Clearly, $\ti{H}$ still achieves its maximum at $(x_{0},t_{0})$. To prove Theorem \ref{Main estimate}, it suffices to prove $u_{11}(x_{0},t_{0})\leq C$. By the maximum principle and (\ref{Definition of dQ}), at $(x_{0},t_{0})$, we have
\begin{equation}\label{Main estimate equation 2}
0 \geq dQ(\ti{H}) = dQ(u_{11})+dQ(|\nabla u|^{2})+2AB_{u},
\end{equation}
where $B_{u}=\Delta u-b|\nabla u|^{2}+a(x)$.

From now on, all the calculations will be carried out at $(x_{0},t_{0})$. For the first term of (\ref{Main estimate equation 2}), using (\ref{Definition of dQ}), we compute
\begin{equation}\label{Main estimate equation 3}
dQ(u_{11}) = u_{tt}\left(\Delta(u_{11})-2b(\nabla u,\nabla u_{11})\right)+B_{u}u_{11tt}-2(\nabla u_{t},\nabla u_{11t}).
\end{equation}
Applying $\nabla_{e_{1}}\nabla_{e_{1}}$ to the equation $G(r)=\log f$ (the logarithm of (\ref{Chen-He equation 1})) and using the concavity of $G$ (see \cite{Donaldson10,CH11,CH18}), we see that
\begin{equation}\label{Main estimate equation 8}
G^{i}(r_{i})_{11}
  =  -G^{i,j}(r_{i})_{1}(r_{j})_{1}+\frac{f_{11}}{f}-\frac{|f_{1}|^{2}}{f^{2}}
\geq \frac{f_{11}}{f}-\frac{|f_{1}|^{2}}{f^{2}},
\end{equation}
where $r=(u_{tt},B_{u},\nabla_{i}u_{t})$. To obtain a lower bound for $G^{i}(r_{i})_{11}$, we need the following lemma.

\begin{lemma}[Lemma 3.1 of \cite{Blocki03}]\label{Lemma}
Let $\Omega$ be a domain in $\mathbf{R}^{n}$ and let $\psi\in C^{1,1}(\ov{\Omega})$ be nonnegative. Then $\sqrt{\psi}\in C^{0,1}(\Omega)$ and
\begin{equation*}
|(D\sqrt{\psi})(x)| \leq
\max\left\{ \frac{|D\psi(x)|}{2\textrm{dist}(x,\de\Omega)},\frac{1+\sup_{\Omega}\lambda_{\textrm{max}}(D^{2}\psi)}{2} \right\}
\end{equation*}
for almost all $x\in\Omega$.
\end{lemma}

Using $\de M=\emptyset$ and Lemma \ref{Lemma} (taking $\psi=f^{\frac{1}{2}}$), we obtain
\begin{equation*}
|\nabla f^{\frac{1}{4}}| \leq C|\nabla(f^{\frac{1}{2}})|+C|\nabla^{2}(f^{\frac{1}{2}})|+C,
\end{equation*}
which implies
\begin{equation*}
|\nabla f|^{2} \leq Cf^{\frac{3}{2}}.
\end{equation*}
Combining this with (\ref{Main estimate equation 8}), it is clear that
\begin{equation*}
G^{i}(r_{i})_{11}
\geq \frac{2(f^{\frac{1}{2}})_{11}}{f^{\frac{1}{2}}}-\frac{|f_{1}|^{2}}{2f^{2}}
\geq -\frac{2|\nabla^{2}(f^{\frac{1}{2}})|}{f^{\frac{1}{2}}}-\frac{|\nabla f|^{2}}{2f^{2}}
\geq -\frac{C}{f^{\frac{1}{2}}}.
\end{equation*}
Recalling that $G(r)=\log Q(r)$ and $Q(r)=f$ (see (\ref{Chen-He equation 1})), it follows that
\begin{equation}\label{Main estimate equation 4}
Q^{i}(r_{i})_{11}
  =  Q(r)G^{i}(r_{i})_{11}
  =  fG^{i}(r_{i})_{11}
\geq -C\sqrt{f}.
\end{equation}
By the commutation formula for covariant derivatives, $r=(u_{tt},B_{u},u_{ti})$, $B_{u}=\Delta u-b|\nabla u|^{2}+a(x)$, $u_{tt}>0$ and $b\geq0$, it is clear that
\begin{equation}\label{Main estimate equation 5}
\begin{split}
        & Q^{i}(r_{i})_{11} \\
  =  {} & u_{tt}(B_{u})_{11}+B_{u}u_{tt11}-2\sum_{i=1}^{n}u_{ti}u_{ti11} \\
  =  {} & u_{tt}\left((\Delta u)_{11}-b(|\nabla u|^{2})_{11}+a_{11}\right)+B_{u}u_{tt11}-2\sum_{i=1}^{n}u_{ti}u_{ti11} \\[2mm]
\leq {} & u_{tt}\left(\Delta(u_{11})+C|\nabla^{2}u|\right)\\[1mm]
        & -bu_{tt}\left(\sum_{i=1}^{n}|u_{i1}|^{2}+2(\nabla u,\nabla u_{11})-C|\nabla u|^{2}\right) \\[1mm]
        & +u_{tt}a_{11}+B_{u}u_{11tt}-2(\nabla u_{11t},\nabla u_{t})+C|\nabla u_{t}|^{2} \\[3mm]
\leq {} & dQ(u_{11})+Cu_{tt}(|\nabla^{2}u|+1)+C|\nabla u_{t}|^{2}, \\[1mm]
\end{split}
\end{equation}
where we used (\ref{Main estimate equation 3}) in the last inequality. Combining (\ref{Main estimate equation 4}) and (\ref{Main estimate equation 5}), we obtain
\begin{equation}\label{Main estimate equation 6}
dQ(u_{11}) \geq -Cu_{tt}(|\nabla^{2}u|+1)-C|\nabla u_{t}|^{2}-C\sqrt{f}.
\end{equation}

For the second term of (\ref{Main estimate equation 2}), by \cite[Proposition 2.9]{CH18}, we have
\begin{equation}\label{Main estimate equation 11}
\begin{split}
dQ(|\nabla u|^{2})
   = {} & 2u_{tt}\left(\textrm{Ric}(\nabla u,\nabla u)-(\nabla u,\nabla a)\right)+2(\nabla f,\nabla u) \\
        & +2u_{tt}|\nabla^{2}u|^{2}+2B_{u}|\nabla u_{t}|^{2}-4\sum_{i,j=1}^{n}u_{ti}u_{tj}u_{ij}.
\end{split}
\end{equation}
For the reader's convenience, we give a proof of (\ref{Main estimate equation 11}) here. Using (\ref{Definition of dQ}), we compute
\begin{equation}\label{Main estimate equation 12}
\begin{split}
dQ(|\nabla u|^{2})
= {} & u_{tt}\left(\Delta(|\nabla u|^{2})-2b(\nabla u,\nabla(|\nabla u|^{2}))\right) \\[2.5mm]
     & +B_{u}(|\nabla u|^{2})_{tt}-2\left(\nabla u_{t},\nabla(|\nabla u|^{2})_{t}\right) \\[2.5mm]
= {} & 2u_{tt}\left(|\nabla^{2}u|^{2}+(\nabla u,\Delta\nabla u)+\textrm{Ric}(\nabla u,\nabla u)\right) \\[2.5mm]
     & -2bu_{tt}\left(\nabla u,\nabla(|\nabla u|^{2})\right)+2B_{u}(\nabla u,\nabla u_{tt})+2B_{u}|\nabla u_{t}|^{2} \\
     & -2\left(\nabla u,\nabla(|\nabla u_{t}|^{2})\right)-4\sum_{i,j=1}^{n}u_{ti}u_{tj}u_{ij},
\end{split}
\end{equation}
where for the second equality, we used
\begin{equation*}
\begin{split}
\left(\nabla u_{t},\nabla(|\nabla u|^{2})_{t}\right)
= {} & 2\sum_{i,j=1}^{n}u_{ti}u_{j}u_{jti}+2\sum_{i,j=1}^{n}u_{ti}u_{ji}u_{jt} \\
= {} & \left(\nabla u,\nabla(|\nabla u_{t}|^{2})\right)+2\sum_{i,j=1}^{n}u_{ti}u_{tj}u_{ij}.
\end{split}
\end{equation*}
Taking derivative of the equation (\ref{Chen-He equation 1}), it is clear that
\begin{equation*}
u_{tt}\left(\nabla\Delta u-b\nabla(|\nabla u|^{2})+\nabla a\right)+B_{u}\nabla u_{tt}-\nabla(|\nabla u_{t}|^{2}) = \nabla f,
\end{equation*}
which implies
\begin{equation}\label{Main estimate equation 13}
\begin{split}
2(\nabla u,\nabla f)-2u_{tt}(\nabla u,\nabla a)
= {} & 2u_{tt}(\nabla u,\nabla\Delta u)-2bu_{tt}(\nabla u,\nabla(|\nabla u|^{2})) \\
     & +2B_{u}(\nabla u,\nabla u_{tt})-2(\nabla u,\nabla(|\nabla u_{t}|^{2})).
\end{split}
\end{equation}
Combining (\ref{Main estimate equation 12}) with (\ref{Main estimate equation 13}), we obtain (\ref{Main estimate equation 11}).

Using (\ref{Main estimate equation 11}) and $u_{tt}>0$, we have
\begin{equation}\label{Main estimate equation 9}
\begin{split}
dQ(|\nabla u|^{2})
\geq {} & -Cu_{tt}-C|\nabla f|+2u_{tt}|\nabla^{2}u|^{2}+2B_{u}|\nabla u_{t}|^{2} \\[2mm]
        & -4n^{2}|\nabla u_{t}|^{2}|\nabla^{2}u|, \\[1mm]
\end{split}
\end{equation}
Recalling the equation (\ref{Chen-He equation 1}) and $f>0$, we have
\begin{equation*}
|\nabla u_{t}|=\sqrt{u_{tt}B_{u}-f}\leq\sqrt{u_{tt}B_{u}},
\end{equation*}
which implies
\begin{equation}\label{Main estimate equation 10}
\begin{split}
4n^{2}|\nabla u_{t}|^{2}|\nabla^{2}u|
\leq {} & 4n^{2}(\sqrt{u_{tt}}|\nabla^{2}u|)(\sqrt{B_{u}}|\nabla u_{t}|) \\
\leq {} & u_{tt}|\nabla^{2}u|^{2}+4n^{4}B_{u}|\nabla u_{t}|^{2}.
\end{split}
\end{equation}
Combining (\ref{Main estimate equation 9}) and (\ref{Main estimate equation 10}), it follows that
\begin{equation}\label{Main estimate equation 7}
\begin{split}
dQ(|\nabla u|^{2})
\geq {} & -Cu_{tt}-C|\nabla f|+u_{tt}|\nabla^{2}u|^{2}-CB_{u}|\nabla u_{t}|^{2} \\
\geq {} & -Cu_{tt}-Cf^{\frac{1}{2}}|\nabla(f^{\frac{1}{2}})|+u_{tt}|\nabla^{2}u|^{2}-CB_{u}|\nabla u_{t}|^{2} \\
\geq {} & u_{tt}(|\nabla^{2}u|^{2}-C)-C|\nabla u_{t}|^{2}-C\sqrt{f},
\end{split}
\end{equation}
where we used $B_{u}\leq C$ in the last inequality. Substituting (\ref{Main estimate equation 6}) and (\ref{Main estimate equation 7}) into (\ref{Main estimate equation 2}), at $(x_{0},t_{0})$, we obtain
\begin{equation}\label{Main estimate equation 14}
0 \geq u_{tt}(|\nabla^{2}u|^{2}-C|\nabla^{2}u|-C)-C|\nabla u_{t}|^{2}-C\sqrt{f}+2AB_{u}.
\end{equation}
From the equation (\ref{Chen-He equation 1}) and $|B_{u}|+|u_{tt}|\leq C$, we have
\begin{equation}\label{Main estimate equation 15}
C|\nabla u_{t}|^{2}+C\sqrt{f}
\leq Cu_{tt}B_{u}+C\sqrt{u_{tt}B_{u}}
\leq C\sqrt{u_{tt}B_{u}}
\leq CB_{u}+Cu_{tt}.
\end{equation}
Substituting (\ref{Main estimate equation 15}) into (\ref{Main estimate equation 14}), it follows that
\begin{equation*}
0 \geq u_{tt}(|\nabla^{2}u|^{2}-C|\nabla^{2}u|-C)+(2A-C)B_{u}.
\end{equation*}
Since $u_{tt}>0$ and $B_{u}>0$, after choosing $A$ sufficiently large, we obtain $u_{11}(x_{0},t_{0})\leq C$, as desired.
\end{proof}

\end{document}